\documentclass[12pt]{article}
\usepackage{amsthm}
\usepackage{latexsym,amssymb,amsmath}
\usepackage{graphicx,float,color,fancybox,shapepar,setspace,hyperref}

\usepackage{subfigure}
\usepackage{pgf,tikz}
\usepackage{enumerate}
\usepackage{subfigure}
\usepackage{tikz}
\usetikzlibrary{decorations.pathreplacing}
\usetikzlibrary{arrows,calc}
\newtheorem{thm}{Theorem}[section]

\newtheorem{fact}[]{Fact}

\newtheorem{claim}[]{Claim}
\newtheorem{assm}[]{Assumption}
\newtheorem{lem}[thm]{Lemma}

\DeclareMathOperator{\ex}{ex}

\newcommand{\p}{\mathcal{P}}

\begin{document}
\begin{sloppypar}


\title{Planar Tur\'an number of quasi-double stars}
\author{{{\small\bf Huiqing  LIU}\thanks{Partially supported by
NNSFC (Nos. 12371345 and 12471325); email: hqliu@hubu.edu.cn;}}\quad{{\small\bf Tian XIE}\thanks{Partially supported by
NNSFC (No. 12371345); email: 13087308856@163.com;}}\quad{{\small\bf Qin ZHAO}\thanks{Partially supported by
NNSFC (No. 12371345) and the Department of Science and Technology of Hubei Province (No.2022CFB871); email: qzhao72@163.com.}}\\ {\small Hubei Key Laboratory of Applied Mathematics,} \\ {\small Faculty of Mathematics and Statistics, Hubei University, Wuhan 430062, China.}\\{\small Key Laboratory of Intelligent Sensing System and Security (Hubei University), } \\ {\small Ministry of Education.}}

\date{}
\maketitle
 \baselineskip17.5pt

\begin{abstract}

\vskip 0.2cm

Given a graph $H$, we call a graph \textit{$H$-free} if it does not contain $H$ as a subgraph.
The \emph{planar Tur\'an number} of a graph $H$, denoted by $\ex_{\p}(n, H)$, is the maximum number of edges in a planar $H$-free graph on $n$ vertices.
A $(h,k)$-quasi-double star $W_{h,k}$, obtained from a path $P_3=v_1v_2v_3$ by adding $h$ leaves and $k$ leaves to the vertices $v_1$ and $v_3$, respectively, is a subclass of caterpillars. In this paper, we study  $\ex_{\mathcal{P}}(n,W_{h,k})$ for all $1\le h\le 2\le k\le 5$, and obtain some tight bounds $\ex_{\mathcal{P}}(n,W_{h,k})\leq\frac{3(h+k)}{h+k+2}n$ for $3\le h+k\le 5$ with equality holds if $(h+k+2)\mid n$, and $\ex_{\mathcal{P}}(n,W_{1,5})\le \frac{5}{2}n$ with equality holds if $12| n$. Also we show that $\frac{9}{4}n\le \ex_{\mathcal{P}}(n,W_{2,4})\le \frac{5}{2}n$ and $\frac{5}{2}n\le\ex_{\mathcal{P}}(n,W_{2,5})\le \frac{17}{6}n$, respectively.
\vskip 0.2cm
			
{\bf Keywords:}  Tur\'an number, planar graph, quasi-double star
			
\end{abstract}
		
\section{Introduction}

All graphs $G=(V(G),E(G))$ considered here are simple and finite, where $V(G)$ and $E(G)$ are vertex and edge set of $G$, respectively.  For two vertices $u, v\in V(G)$, the \textit{distance} between $u$ and $v$ is denoted by $\mathrm{dist}_G(u,v)$. Let $N_k(v)=\{ u\in V(G)\ | \ \mathrm{dist}_G(u,v)=k\}$. For simplicity, we use the notation $N(v)$ instead of $N_1(v)$, and set $d(v)=|N(v)|$, $N[v]=N(v)\cup \{v\}$. We denote by $\delta(G)$ and $\Delta(G)$ the \textit{minimum degree} and \textit{maximum degree} of $G$, respectively. A vertex of degree one is called a \textit{leaf}. For $S\subseteq V(G)$, the subgraph induced by $S$ is denoted by $G[S]$. Let $G-S=G[V(G)\backslash S]$. If $S=\{v\}$, then we simplify $G-\{v\}$ to $G-v$. For any $S$, $T\subset V(G)$, we use $E[S,T]$ to denote the edges with one end in $S$ and the other in $T$. Let $e(G)=|E(G)|$ and $e[S,T]=|E[S,T]|$.

By $ C_n, P_n, K_{s,n-s}$ we denote a cycle, a path, a complete bipartite graph on $n$ vertices, respectively. A vertex of degree $n-1$ in a star $K_{1,n-1}$ is called its \textit{center}.
A $p$-$q$ edge is an edge whose two end vertices are of degrees $p$ and $q$, respectively.
A $p$-$l$-$q$-path is a path $P_3$ whose vertices are of degrees $p$, $l$ and $q$, respectively.
A $p$-$q$ star with center $x$ is a star $K_{1,p}$ with center $x$ whose $p$ leaves all are of degree $q$.

Given a graph $H$, we call a graph \textit{$H$-free} if it does not contain $H$ as a subgraph.
The \emph{planar Tur\'an number} of a graph $H$, denoted by $\ex_{\p}(n, H)$, is the maximum number of edges
in a planar $H$-free graph on $n$ vertices. In 2016, Dowden  \cite{dowden2016extremal}  obtained a tight bound  $\ex_{\mathcal{P}}(n,C_{4})\le \frac{15}{7}(n-2)$ for $n\ge 4$ and $\ex_{\mathcal{P}}(n,C_{5})\le \frac{12}{5}n-\frac{33}{5}$ for all $n\ge 11$. In 2022, Ghosh {\em et al.} \cite{ghosh2022planar} determined the bound  $\ex_{\mathcal{P}}(n,C_{6})\leq \frac{5}{2}n-7$ for all $n\geq 18$. Later, Shi, Ruilin and Walsh \cite{shi2025planar}, and independently Gy{\H{o}}ri, Li and Zhou \cite{gyHori2023planar} showed that $\ex_{\mathcal{P}}(n,C_{7})\le\frac{18}{7}n-\frac{48}{7}$.


\emph{Caterpillars} are an important subclass of trees, first studied by Harary and Schwenk \cite{harary1973number}, in which the removal of all leaves makes it a path. Let $P_l(s_1,s_2,\dots,s_l)$ be a caterpillar obtained from $P_l:=v_1v_2\cdots v_l$ by adding $s_i$ leaves to the vertex $v_i$ for $1\le i\le l$. Denote by $S_{h,k}:=P_2(h,k)$ \emph{a double star}, and by $W_{h,k}:=P_3(h,0,k)$ \emph{a quasi-double star}. If $u_i\in N(v_1)\setminus\{v_2\}$ for $1\le i\le h$ and $w_j\in N(v_3)\setminus\{v_2\}$ for $1\le j\le k$, then we write $(u_1u_2\dots u_h)v_1$-$v_2$-$v_3(w_1w_2\dots w_k)$ for $W_{h,k}$, see Figure~\ref{fig-smn}.

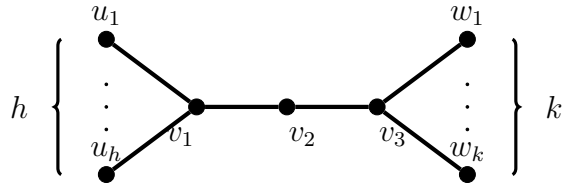
\begin{figure}[htbp]
\centering

\begin{tikzpicture}[scale=0.6]
\tikzstyle{vertex}=[draw,circle,fill=black,minimum size=6,inner sep=0]		
		
\node[vertex] (u_1) at (-4,1.5)[label={[yshift=-2pt] $u_1$}] {}circle(15pt);
\node[vertex] (u_h) at
(-4,-1.5)[label={[yshift=-2pt] $u_h$}] {};
\node[vertex] (w_1) at (4,1.5)[label={[yshift=-2pt] $w_1$}] {};
\node[vertex] (w_k) at (4,-1.5)[label={[yshift=-2pt] $w_k$}] {};
\node[vertex] (u) at (-2,0)[label={[xshift=-6pt, yshift=-22pt] $v_1$}] {};
\node[vertex] (v) at (0,0)[label={[xshift=6pt, yshift=-22pt] $v_2$}] {};
\node[vertex] (w) at (2,0)[label={[xshift=6pt, yshift=-22pt] $v_3$}] {};

\draw[line width=1.2pt,decorate, decoration={brace}] (-5,-1.5) -- (-5,1.5)
node[midway, left=8pt] {$h$};	
\draw[line width=1.2pt,decorate, decoration={brace,mirror}] (5,-1.5) -- (5,1.5)
node[midway, right=8pt] {$k$};	
\draw[fill=black] (-4,0.5) circle (0.8pt);
\draw[fill=black] (-4,-0.5) circle (0.8pt);
\draw[fill=black] (-4,0) circle (0.8pt);
\draw[fill=black] (4,0.5) circle (0.8pt);
\draw[fill=black] (4,-0.5) circle (0.8pt);
\draw[fill=black] (4,0) circle(0.8pt);		
\draw[ultra thick] (u) -- (u_1);
\draw[ultra thick] (u) -- (u_h);
\draw[ultra thick] (w) -- (w_1);
\draw[ultra thick] (w) -- (w_k);
\draw[ultra thick] (u) -- (v);
\draw[ultra thick] (w) -- (v);
	\end{tikzpicture}	
\caption{Quasi-double star $W_{h,k}$}\label{fig-smn}
\end{figure}
 Recently, Ghosh {\em et al.} \cite{ghosh2021planar}  considered $\ex_{\p}(n, H)$ for $H$ being double stars.

\begin{thm}\cite{ghosh2021planar}
Estimates on the planar Tur\'an number of double stars are as follows:

\begin{enumerate}
    \item For $n\geq 16$, $\ex_{\p}(n,S_{2,2})= 2n-4$.
    \item For $n\geq 1$, $\ex_{\mathcal{P}}(n,S_{2,3})=2n.$
    \item For $n\geq 1$, $\frac{15}{7}n \leq \ex_{\mathcal{P}}(n,S_{2,4})\leq\frac{8}{3}n.$
    \item For $n\geq 1$, $\frac{5}{2}n \leq \ex_{\mathcal{P}}(n,S_{2,5})\leq\frac{17}{6}n.$
    \item For $n\geq 3$, $\frac{5}{2}n-5\leq\ex_{\mathcal{P}}(n,S_{3,3})\leq\frac{5}{2}n-2.$
    \item For $n\geq 1$, $\frac{9}{4}n\leq \ex_{\mathcal{P}}(n,S_{3,4})\leq\frac{17}{6}n.$
\end{enumerate}
\end{thm}

Later, Xu and Shao~\cite{xu2025planar} improved the upper bound  $\ex_{\mathcal{P}}(n,S_{2,4})\le\frac{31}{14}n$ for $n\ge1$. Xu, Hu and Zhang~\cite{xu2024improved} showed that  $\ex_{\mathcal{P}}(n,S_{2,5})\le\frac{19}{7}n-\frac{18}{7}$ for $n\ge1$. Xu, Zhou, Li and Yan~\cite{xu2024planarbds} gave exact value of $\ex_{\mathcal{P}}(n,S_{3,3})=\lfloor\frac{5}{2}n\rfloor-5$ for $n\ge10$. Xu, Zhang and Shao~\cite{xu2024planarsS34} showed that  $\ex_{\mathcal{P}}(n,S_{3,4})\le\frac{5}{2}n$ for $n\ge1$. Moreover, Liu and Xu~\cite{liu2025upper} established an upper bound for $\ex_{\mathcal{P}}(n,S_{3,5})\le\frac{23}{8}n-3$ for $n\ge2$. For more results on planar Tur\'an number, we refer to~\cite{fang2024extremal}, \cite{fang2022planar}, \cite{guan2024planar}, \cite{lan2019planar}, \cite{lan2024planar}.\medskip

In this paper, we determine $\ex_{\p}(n, H)$ for $H$ being quasi-double stars $W_{h,k}$. Our main results are as follows.

\begin{thm}\label{thm-h+k=6}
	Let $G$ be a planar $W_{h,k}$-free graph on $n$ vertices, where $1\le h\le2\le k\le 5$.
\begin{enumerate}

     \item If $3\le h+k\le 5$, then $\ex_{\mathcal{P}}(n,W_{h,k})\le \frac{3(h+k)}{h+k+2}n$ and equality holds if $(h+k+2) | n$.

    \item If $h+k=6$, then $\frac{9}{4}n\le \ex_{\mathcal{P}}(n,W_{h,k})\le \frac{5}{2}n$. Moreover, right equality holds if $12 | n$ and $h=1,~k=5$.

    \item $\frac{5}{2}n\le \ex_{\mathcal{P}}(n,W_{2,5})\le \frac{17}{6}n$.
\end{enumerate}	
	
\end{thm}

\section{Preliminaries}
In this section, we first reveal some local structure of a planar $W_{h,k}$-free graph.

\begin{lem}\label{thm-edge}\cite{bondy2008graph}
Let $G$ be a simple planar graph on $n$ ($n\ge 3$) vertices, then $e(G)\le 3n-6$. Moreover, if $G$ is bipartite, then $e(G)\le 2n-4$.
\end{lem}

\begin{lem}\label{obs-s1k}
	Let $G$ be a planar $W_{h,k}$-free graph on $n$ vertices with $\delta(G)\geq h+ 1$, and let $x\in V(G)$. If $d(x)\ge h+k+1$, then $G[N[x]]$ is a component of $G$. Moreover, if $d(x)\geq h+k+2$, then $G[N(x)]$ is $K_{1,h+1}$-free.
\end{lem}

\begin{proof} Let $N(x)=\{x_{1},x_{2},\dots,x_t\}$, where $t=d(x)$. Suppose that $G[N[x]]$ is not a component of $G$, then $E(N[x], V(G)\setminus N[x])\neq\emptyset$, and thus there exist some $x_i\in N(x)$, say $x_1$, and $y\in V(G)\setminus N[x]$ such that $x_1y\in E(G)$. Since $d(y)\ge\delta(G)\ge h+ 1$, we can let $y_i\in N(y)\setminus\{x_1\}$ for $1\le i\le h$ and $Y=\{y_1,\ldots,y_h\}$. Note that $t\ge h+k+1$, and hence $|N(x)\setminus Y|\ge k$. So we can assume $Y\cap \{x_{h+ 2},\ldots,x_{h+k+1}\}=\emptyset$. Then $(y_1y_2\dots y_h)y$-$x_1$-$x(x_{h+ 2}\cdots x_{h+k+1})$ is a copy of $W_{h,k}$, a contradiction.
	
	Next we suppose that $d(x)\ge h+k+2$ and $G[N(x)]$ contains a $K_{1,h+ 1}$ as its subgraph. Without loss of generality, we can assume that $G'$ is the  star $K_{1,h+ 1}$ with the bipartition $\{x_1\}$ and $\{x_2,\dots,x_{h+ 2}\}$, then $(x_{3},\dots,x_{h+ 2})x_{1}$-$x_{2}$-$x(x_{h+ 3}\dots x_{h+k+2})$ is a copy of  $W_{h,k}$, again a contradiction.\end{proof}

\begin{lem}\label{obs-xyz}
	Let $G$ be a planar $W_{h,k}$-free graph on $n$ vertices and $\delta(G)\geq h+ 1$, and let $x\in V(G)$. If $d(x)\ge h+k$, then $N(y)\subseteq N(x)$ for any $y\in N_2(x)$. Furthermore, $G[N[x]\cup N_2(x)]$ is a component of $G$.
\end{lem}
\begin{proof}
	Let $N(x)=\{x_{1},x_{2},\dots,x_t\}$, and let $xx_1y$ be a path $P_3$, where $t=d(x)$ and $y\in N_2(x)$. Assume, to the contrary that $N(y)\nsubseteq N(x)$, that is, there exists a vertex $y_1\in N(y)\setminus N(x)$. Since $d(y)\ge\delta(G)\ge  h+ 1$, we can let $y_i\in N(y)\setminus\{x_1,y_1\}$ for $2\le i\le h$. Note that $t\ge h+k$ and $y_1\notin N(x)$, and hence we can assume that $\{y_1,y_2\dots y_h\}\cap \{x_{h+ 1},\dots,x_{h+k}\}=\emptyset$. So $(y_1y_2\dots y_h)y$-$x_1$-$x(x_{h+ 1}\dots x_{h+k})$ is a copy of $W_{h,k}$, a contradiction.
\end{proof}

\begin{lem}\label{obs-star}
	Let $G$ be a planar $W_{h,k}$-free graph on $n$ vertices and $\delta(G)\geq h+ 1$, where $ h+k\ge5$. If $G$ contains a $(h+k)$-$(h+k-1)$ star with center $x$. Then we have $N_2(x)\neq \emptyset$.
\end{lem}
\begin{proof}
	Let $N(x)=\{x_{1},x_{2},\dots,x_{h+k}\}$ with $d(x_i)=h+k-1$ for $1\le i\le h+k$. Suppose that $N_2(x)=\emptyset$. Then $G[N[x]]$ is a component on $h+k+1$ vertices, and then $$e(G[N[x]])=\frac{(h+k)+(h+k-1)(h+k)}{2}=\frac{(h+k)^2}{2}.$$
On the other hand, if $h+k$ is odd, then $d(x)=h+k$ is the unique vertex of odd degree in $G[N[x]]$, a contradiction. So $h+k$ is even, that is, $h+k\ge 6$. Note that $G[N[x]]$ is planar, we have $$e(G[N[x]])\le 3(h+k+1)-6=3k+3h-3<\frac{(h+k)^2}{2}=e(G[N[x]])$$ for $h+k\ge6$, a contradiction.\end{proof}
\begin{lem}\label{obs-xy}
	Let $G$ be a planar $W_{h,k}$-free graph, and let $xy$ be a $p$-$q$ edge with $d(x)=p\ge k+2\ge q\ge d(y)$. If $N(x)\cap N(y)\neq\emptyset$, then $|N(x)\cap N(y)|\ge q-h$.
\end{lem}

\begin{proof}
	Suppose that $|N(x)\cap N(y)|\le q-h-1$. Then $|N(y)\setminus N[x]|=|N(y)\setminus\{x\}|-|N(x)\cap N(y)|\ge (q-1)-(q-h-1)=h, $ that is, there exist $h$ vertices $y_1,y_2,\dots,y_h\in N(y)\setminus N[x]$. Let $x_1\in N(x)\cap N(y) $ and $N(x)\supseteq \{y,x_1,x_2\dots,x_{k+1}\}$, then $(y_1y_2\dots y_h)y$-$x_1$-$x(x_2x_3\dots x_{k+1})$ is a copy of $W_{h,k}$, a contradiction.
\end{proof}

Next we present some results will be used in the proof of our results.
\begin{lem}\label{C-edge}
	Let $G$ be a simple planar graph on $n$ vertices, and let $C$ a component of $G$. If $e(G-C)\le \tau(n-|C|)$ and $|C|\le \frac {6}{3-\tau}$, where $\tau<3$, then $e(G)\le\tau n$. \end{lem}
\begin{proof} Since $G$ is planar, $C$ is planar, and thus, by Lemma \ref{thm-edge}, $e(C)\le 3|C|-6$. Note that
	$C$ is a component of $G$, and hence $$e(G)=e(G-C)+e(G[C])\le \tau(n-|C|)+3|C|-6=\tau n+(3-\tau)|C|-6\le \tau n.$$ So the assertion holds.
\end{proof}
\begin{lem}\label{obs-k+2}
	Let $G$ be a planar $W_{h,k}$-free graph on $n$ vertices and $\delta(G)\ge h+ 1$, where $1\le h\le 2\le k$, and let $x\in V(G)$ with $d(x)\ge h+k+1$. Suppose that $e(G- N[x])\le \frac{3(h+k)}{h+k+2}(n-d(x)-1),$  then $$e(G)\le\frac{3(h+k)}{h+k+2}n.$$
\end{lem}

\begin{proof} Since $d(x)\ge h+k+1$, $G[N[x]]$ is a component on $d(x)+1$ vertices by Lemma  \ref{obs-s1k}. 
If $d(x)=h+k+1$, then $|N[x]|=h+k+2=\frac {6}{3-\tau}$, where $\tau=\frac{3(h+k)}{h+k+2}$, and hence, by Lemma \ref{C-edge},  $e(G)\le\frac{3(h+k)}{h+k+2}n$.
	
If  $d(x)\ge h+k+2$, then by Lemma  \ref{obs-s1k}, $G[N(x)]$ is $K_{1,h+ 1}$-free. Thus $$e(G[N[x]])\leq\frac{1}{2}\left(d(x)+(h+ 1)d(x)\right)=\frac{h+ 2}{2}d(x)\le 2d(x).\eqno(1)$$ So, by the assumption and (1), we have
	\begin{eqnarray*}~~~~~e(G)&=&e(G-N[x]) + e(G[N[x]])\leq \frac{3(h+k)}{h+k+2}(n-d(x)-1)+{2}d(x)\\
		&=&\frac{3(h+k)}{h+k+2}n-\frac{(k-3)d(x)}{h+k+2}-\frac{3(k+1)}{k+3} \\
		&\leq&\frac{3(h+k)}{h+k+2}n-\frac{k^{2}+3k-6}{k+3}<\frac{3(h+k)}{h+k+2}n,\end{eqnarray*}where the last second inequality follows from $d(x)\ge h+k+2$.
\end{proof}

We omit the proof of the next lemma since it is straightforward.
\begin{lem}\label{delta}
	Let $G=(V(G),E(G))$ be a planar graph on $n$ vertices, and let $x\in V(G)$. If $d(x)\le  \sigma$ and $e(G-x)\le \sigma(n-1)$, then $e(G)\le  \sigma n$. \end{lem}

\section{Proof of Theorem~\ref{thm-h+k=6}}


\indent First, we present the lower bound of the $\ex_{\p}(n,W_{h,k})$. Since $W_{h,k}$ contains $h+k+3$ vertices, then a maximal planar graph with $n\leq h+k+2$ vertices does not contain a $W_{h,k}$. Let $(h+k+2)|n$. Consider the plane graph consisting of $\frac{n}{h+k+2}$ disjoint copies of maximal planar graphs on $h+k+2$ vertices. This graph does not contain a $W_{h,k}$. Hence, $$\ex_{\p}(n,W_{h,k})\geq (3(h+k+2)-6) \cdot \frac{n}{h+k+2} =\frac{3(h+k)}{h+k+2}n.$$
In particular, we have $\ex_{\p}(n,W_{2,4})\ge\frac{9}{4}n$.

Next we will show that $\ex_{\p}(n,W_{h,5})\geq\frac{5}{2}n$. To prove the lower bound, we construct the plane graph consisting of $\frac{n}{12}$ disjoint copies of $5$-regular triangulation on $12$ vertices, see Figure \ref{figure-s15}. Note that $W_{h,5}$ contains a vertex which degree is six, so we know that this graph does not contain a $W_{h,5}$. Hence $$\ex_{\p}(n,W_{h,5})\ge\frac{n}{12}\cdot (36-6)=\frac{5}{2}n.$$

\begin{figure}[ht]\medskip
	\centering
	\subfigure{
		\begin{minipage}[t]{0.3\textwidth}
			\centering
			\begin{tikzpicture}[scale=0.7]
				\tikzstyle{vertex}=[draw,circle,fill=black,minimum size=6,inner sep=0]		
\draw[ultra thick] (0,0) circle (2cm);
\draw[fill=black](0, 2)circle(4.0pt);
\draw[fill=black](-1.75,-1)circle(4.0pt);
\draw[fill=black](1.75,-1)circle(4.0pt);
\draw[fill=black](0, 1.3)circle(4.0pt);
\draw[fill=black](0,-1.3)circle(4.0pt);
\draw[fill=black](1.15,0.65)circle(4.0pt);
\draw[fill=black](-1.15,0.65)circle(4.0pt);
\draw[fill=black](-1.15,-0.65)circle(4.0pt);
\draw[fill=black](1.15,-0.65)circle(4.0pt);
\draw[fill=black](0, -0.65)circle(4.0pt);
\draw[fill=black](0.55,0.32)circle(4.0pt);
\draw[fill=black](-0.55,0.32)circle(4.0pt);

\draw[ultra thick](0,2)--(1.15,0.65)--(1.75,-1)--(0,-1.3)--(-1.75,-1)--(-1.15,0.65)--(0,2)(0,1.3)--(1.15,0.65)--(1.15,-0.65)--(0,-1.3)--(-1.15,-0.65)--(-1.15,0.65)--(0,1.3)(0,2)--(0,1.3)(1.75,-1)--(1.15,-0.65)(-1.75,-1)--(-1.15,-0.65)(0,-0.65)--(0.55,0.32)--(-0.55,0.32)--(0,-0.65)(0.55,0.32)--(1.15,0.65)(-0.55,0.32)--(-1.15,0.65)(0,-0.65)--(0,-1.3)(0,1.3)--(0.55,0.32)(0,1.3)--(-0.55,0.32)(0,-0.65)--(1.15,-0.65)--(0.55,0.32)(0,-0.65)--(-1.15,-0.65)--(-0.55,0.32);
			\end{tikzpicture}
			
		\end{minipage}
	}				
	\caption{The $5$-regular triangulation on $12$ vertices and $30$ edges.}
	\label{figure-s15}
\end{figure}
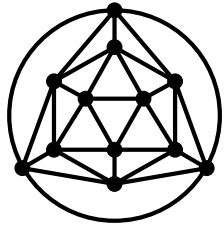

Following we will prove the upper bound of $\ex_{\p}(n,W_{h,k})$ for $1\le h\le 2\le k\le5$.  Let
$$
	f(n)=
		\left\{
		\begin{array}{ll}
			\frac{3(h+k)}{h+k+2}n, &  \text{if } 3\le h+k\le 5;\\
			
			\frac{5}{2}n,  &  \text{if } h+k=6;\\
			
			\frac{17}{6}n,  &  \text{if } h=2,~k=5.
		\end{array}
		\right.
	$$ Note that if $h+k\le 10$, then we can check that $\frac{3(h+k)}{h+k+2}n\le\frac{5}{2}n<\frac{17}{6}n$.

We proceed the proof by induction on $n$.  If $n\le h+k+2$, then $n-2\le n-\frac{2}{h+k+2}n=\frac{(h+k)}{h+k+2}n$, and thus, by Lemma \ref{thm-edge}, $$e(G)\le 3n-6\le \frac{3(h+k)}{h+k+2}n\le f(n)$$ for $1\le h\le 2\le k\le5$. So we can assume that $n\ge h+k+3$ and $\ex_{\p}(l,W_{h,k})\le f(l)$ for $l\le n-1$. We are now going to show that $\ex_{\p}(n,W_{h,k})\le f(n)$.

Let $G$ be a planar $W_{h,k}$-free graph on $n$ vertices  for $1\le h\le 2\le k\le5$, and let $V_i=\{ v\in V(G)~|~ d(v)=i\}$,  $n_{i}=|V_i|$. Then $G$ is $K_{3,3}$-free. For a component $H$ of $G$, we will write $V(H)$ for it.
For any $\emptyset\subset T \subset V(G)$, we denote $\overline{T}:=V(G)\setminus T$.  For any $xy\in E(G)$, we always  let $S[xy]:=N[x]\cup N(y)$, $S_{xy}:= N(x)\cap N(y)$, $S_x:= N(x)\setminus N[y]$ and $S_y:=N(y)\setminus N[x]$. Then $S[xy]=\{x,y\}\cup S_{xy}\cup S_x\cup S_y$, see Figure \ref{fig1}.

\begin{figure}[ht]\medskip
	\centering
	\subfigure{
		\begin{minipage}[t]{0.3\textwidth}
			\centering
\begin{tikzpicture}[scale=0.7]
	\tikzstyle{vertex}=[draw,circle,fill=black,minimum size=6,inner sep=0]		
	
	\node[vertex] (e) at
	(3.7,0){};
	\node[vertex] (f) at
	(3.7,-1){};
	\node[vertex] (d) at (-1.3,-1) {};
	\node[vertex] (c) at (-1.3,0){};
	\node[vertex] (b) at (1.2,0) {};
	\node[vertex] (a) at (1.2,1){};
	\node[vertex] (x) at (0,-1)[label={[xshift=-6pt, yshift=-22pt] $x$}] {};
	\node[vertex] (y) at (2.4,-1)[label={[xshift=6pt, yshift=-22pt] $y$}] {};
	\draw (-2.1,-0.5) node[right=0pt]{$S_x$};
	\draw (3.7,-0.5) node[right=0pt]{$S_y$};
	\draw (0.6,1.5) node[right=0pt]{$S_{xy}$};
	\draw[fill=black] (-1.3,-0.3) circle (0.8pt);
	\draw[fill=black] (-1.3,-0.5) circle (0.8pt);
	\draw[fill=black] (-1.3,-0.7) circle (0.8pt);
	\draw[fill=black] (3.7,-0.3) circle (0.8pt);
	\draw[fill=black] (3.7,-0.5) circle (0.8pt);
	\draw[fill=black] (3.7,-0.7) circle (0.8pt);
	\draw[fill=black] (1.2,0.3) circle (0.8pt);
	\draw[fill=black] (1.2,0.5) circle (0.8pt);
	\draw[fill=black] (1.2,0.7) circle (0.8pt);
	
	\draw[ultra thick] (x) -- (a);
	\draw[ultra thick] (x) -- (b);
	\draw[ultra thick] (x) -- (c);
	\draw[ultra thick] (x) -- (d);
	\draw[ultra thick] (y) -- (a);
	\draw[ultra thick] (y) -- (b);
	\draw[ultra thick] (y) -- (e);
	\draw[ultra thick] (y) -- (f);
	\draw[ultra thick] (x) -- (y);
	\draw[thick, dashed,] (0.6,-0.5)rectangle(1.8,1.8);
	\draw[thick, dashed, blue] (-0.7,0.5)rectangle(-2.0,-1.5);
	\draw[thick, dashed, red] (4.3,0.5)rectangle(3.0,-1.5);
\end{tikzpicture}

\end{minipage}
}				
	\caption{$S[xy]$ for $xy\in E(G)$ }\label{fig1}
\end{figure}
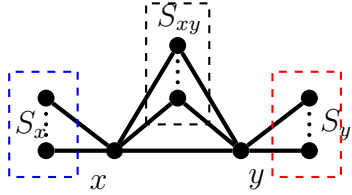

Now we will present two claims.

\begin{claim} \label{S-com}
	 If $G$ contains a $(h+k)$-$(h+k)$ edge $xy$, then $S[xy]$ is a component of $G$.
\end{claim}

\noindent{\bf Proof of Claim \ref{S-com}.}  First we note that if $v\in S_x$, then $v
\in N_2(y)$, and then by Lemma  \ref{obs-xyz}, $N(v)\subseteq N(y)$. Similarly, $N(u)\subseteq N(x)$ for all $u\in S_y$.
 If $S[xy]$ is not a component of $G$, then there exists two vertices $w\in S_{xy}$ and $z\in \overline{S[xy]}$ such that $wz\in E(G)$.
 Note that $z\in N_2(x)\cap N_2(y)$ and $d(x)=d(y)=h+k$, and hence, by Lemma  \ref{obs-xyz}, $N(z)\subseteq S_{xy}$. Thus $|S_{xy}|\ge |N(z)|=d(z)\ge \delta(G)\ge 3$, which implies that  $G$ contains a $K_{3,d(z)}$ with the bipartition $N(z)$ and $\{x,y,z\}$, a contradiction. Therefore, $S[xy]$ is a component of $G$. { \hfill $\Box$}

 \begin{claim} \label{S-3path}
 	If $G$ contains a $(h+k)$-$l$-$(h+k)$ path $xyz$ with $xz\notin E(G)$ and $d(x)=d(z)=h+k$, $d(y)=l$ for $3\le l\le h+k-1$, then $N(x)=N(z)$. Moreover, $N[x]\cup \{z\}$ is a component of $G$.
 \end{claim}

 \noindent{\bf Proof of Claim  \ref{S-3path}.}
 Since $xz\notin E(G)$, we have  $x\in N_2(z)$ and $z\in N_2(x)$. Note that $d(x)=d(z)=h+k$, and hence, by Lemma  \ref{obs-xyz}, $N(x)\subseteq N(z)$ and $N(z)\subseteq N(x)$. So $N(x)=N(z)$.

 Next we show that $N[x]\cup \{z\}$ is a component of $G$.
 If not, then there exist two vertices $u\in\overline{N[x]\cup \{z\}}$ and $v\in N(x)$ such that $uv\in E(G)$, that is, $u\in N_2(x)\cap N_2(z)$. Note that $d(x)=d(z)=h+k$, and hence by Lemma  \ref{obs-xyz}, we have $N(u)\subseteq (N(x)\cap N(z))$. Since $|N(u)|=d(u)\ge\delta(G)\geq 3$, there exists a $K_{3,d(u)}$ with the bipartition $N(u)$ and $\{x,z,u\}$ in $G$, a contradiction. Therefore, $N[x]\cup \{z\}$ is a component of $G$.
 { \hfill $\Box$}\medskip

If $G$ has a vertex $v$ with $d(v)\leq 1<9/5=\frac{3(h+k)}{h+k+2}$ when $h+k=3$, and $d(v)\leq 2\le\frac{3(h+k)}{h+k+2}$ when $h+k\ge 4$, then by the induction hypothesis, $e(G-x)\le \frac{3(h+k)}{h+k+2}(n-1)$, and thus, by Lemma \ref{delta}, $e(G)\le \frac{3(h+k)}{h+k+2}n\le f(n).$
 So we can assume that $\delta(G)\geq 2=h+1$ when $h+k=3$ (here $h=1$), and $\delta(G)\geq 3\ge h+1$ when $h+k\ge 4$. Furthermore, if $G$ contains a vertex $x$ such that $d(x)\ge h+k+1$, then by Lemma  \ref{obs-s1k}, $N[x]$ is a component of $G$, and then by the induction hypothesis, $e(G- N[x])\le \frac{3(h+k)}{h+k+2}(n-d(x)-1)$, and thus, by Lemma \ref{obs-k+2}, $e(G)\leq \frac{3(h+k)}{h+k+2}n\le f(n).$ \medskip

 Therefore we may assume $\Delta(G) \le h+k$. Moreover, if $h+k\le 4$, then $$e(G)\le \frac{n\Delta(G)}{2}\le \frac{n(h+k)}{2}\le   \frac{3n(h+k)}{h+k+2}\le f(n).$$ Hence, in the following, we always assume that $h+k\ge 5$. 

\subsection{ $h+k=5$ }
In this subsection, $(h,k)\in \{(1,4),(2,3)\}$, $f(n)=\frac{15}{7}n$, $\delta(G)\ge 3$ and $\Delta(G)\le  5$.  If $\Delta(G)\le 4$, then  $e(G)\le 2n< \frac{15}{7}n$, and if $n\le 7$, then $e(G)\le 3n-6\le \frac{15}{7}n$ clearly. So we assume that $n\ge8$ and $\Delta(G)= 5$.

 \begin{fact}\label{lem-s14}
	 If $G$ contains any copy of $5$-$5$ edge, then $e(G)\leq \frac{15}{7}n$.
\end{fact}

\noindent{\bf Proof of Fact \ref{lem-s14}.}  Let $xy$ be a $5$-$5$ edge in $G$. Then by Claim  \ref{S-com}, $S[xy]$ is a component of $G$. Then $e(G)=e(G-{S[xy]})+e(G[S[xy]])$. If $S[xy]\subset V(G)$, then by the induction hypothesis, we have $$
	e(G-{S[xy]})\le \frac{15}{7}(n-10).\eqno(2)$$

If $S_{xy}=\emptyset$, then $|S[xy]|=10$ and $|S_x|=|S_y|=4$. Since $G$ is $W_{h,k}$-free, $G[S_x]$ and $G[S_y]$ are $K_{1,h}$-free for $h\le 2$, and then $e(G[S_x])\le {|S_x|}/{2}=2$, $e(G[S_y])\le {|S_y|}/{2}=2$. Let $G'$ be the bipartite subgraph of $G[S[xy]]$ with the bipartition $S_x\cup\{y\}$ and $S_y\cup\{x\}$, then by Lemma \ref{thm-edge}, $e(G')\le 2(|S_x|+|S_y|+2)-4=16.$ So
$$e(G[S[xy]])=e(G[S_x])+e(G[S_y])+e(G')\le 2+2+16=20.\eqno(3)$$
If $S[xy]=V(G)$, then $n=|S[xy]|=10$, and then by (3), $e(G)\le20<\frac{150}{7}=\frac{15}{7}n.$
	So we can assume that $S[xy]\subset V(G)$. Then by (2) and (3), we have  $$e(G)=e(G-{S[xy]})+e(G[S[xy]])\leq\frac{15}{7}(n-10)+20=\frac{15}{7}n-\frac{10}{7}<\frac{15}{7}n.$$
So, in the following, we can assume that $S_{xy}\neq\emptyset$. Choose the $5$-$5$ edge $xy$ such that $|S_{xy}|$ is as large as possible. If $|S_{xy}|\ge3$, then
$$|S[xy]|=10-|S_{xy}|\le 7= \frac{6}{3-15/7}<8\le n~ (\tau:=15/7),\eqno(4)$$ which implies that $S[xy]\subset V(G)$, and then by induction hypothesis, $e(G-S[xy])\le \frac{15}{7}(n-|S[xy]|)$. Thus, by Lemma \ref{C-edge} and (4), we have $e(G)\le\frac{15}{7}n=f(n)$. \medskip

 Now we only need to consider the case that $1\le |S_{xy}|\le2$, then $h=1$ and $k=4$ (for otherwise, $h=2$ and $k=3$, and then by Lemma  \ref{obs-xy}, we have $|S_{xy}|\ge d(y)-2=3$, a contradiction). Furthermore, $S_x$ and $S_y$ are independent as $G$ is $W_{1,4}$-free, and then by Lemma \ref{thm-edge}, $e[S_x,S_y]=e(G[S_x\cup S_y])\le 2|S_x\cup S_y|-4=4(3-|S_{xy}|)$. Moreover, if $|S_{xy}|=2$, $|S_x|=|S_y|=2$, then $e[S_x,S_y]\le 3$ (for otherwise, $e[S_x,S_y]=4$, which implies there exists a $K_{3,3}$ with the bipartition $S_x\cup \{y\}$ and $S_y\cup \{x\}$, a contradiction). So $$e[S_x,S_y]\le 4(3-|S_{xy}|)-|S_{xy}|+1=13-5|S_{xy}|.\eqno(5)$$
Next we calim that $$e[S_{xy}, S_x\cup S_y]\le 2|S_{xy}|.\eqno(6)$$
 Suppose not, then there exists $z\in S_{xy}$ such that $|N(z)\cap ({S_x\cup S_y})|\ge 3$. So $d(z)\ge 5$, furthermore, $d(z)= 5$ as $\Delta(G)=5$, that is $xz$ is a $5$-$5$ edge. Since $|N(z)\cap ({S_x\cup S_y})|\ge 3$,  we can assume  $|N(z)\cap {S_x}|\ge 2$, then $|N(z)\cap {N(x)}|\ge 3$, that is $|S_{xz}|\ge3>2\ge |S_{xy}|$, a contradiction with the choice of $xy$.\medskip

Note that $|S_{xy}|\le2$, then $e(G[S_{xy}])\le |S_{xy}|-1$. Thus by (2), (5) and (6),
 \begin{align*}
	e(G)
	&=e(G-{S[xy]})+9+e(G[S_{xy}])+e[S_{xy}, S_x\cup S_y]+e[S_x,S_y]\\
	&\leq\frac{15}{7}(n-(10-|S_{xy}|))+9+(|S_{xy}|-1)+2|S_{xy}|+(13-5|S_{xy}|)\\
	&=\frac{15}{7}n-\frac{3-|S_{xy}|}{7}<\frac{15}{7}n-\frac{1}{7}<\frac{15}{7}n.
\end{align*}Therefore, the proof of Fact \ref{lem-s14} is complete. { \hfill $\Box$}

\begin{fact}\label{lem-5k5}
	If $G$ contains any copy of $5$-$l$-$5$ path for $3\le l\le 4$, then $e(G)\leq \frac{15}{7}n$.
\end{fact}

\noindent{\bf Proof of Fact \ref{lem-5k5}}.  Let $xyz$ be a $5$-$l$-$5$ path with $d(x)=d(z)=5$ and $d(y)=l$ for $3\le l\le 4$. By Fact \ref{lem-s14}, we can suppose that $xz\notin E(G)$. Then by Claim  \ref{S-3path}, we have $N(x)=N(z)$ and $N[x]\cup \{z\}$ is a component of $G$. Note that $|N[x]\cup \{z\}|=7=\frac{6}{3-15/7}<8\le n$ ($\tau:=15/7$), that is ${N[x]\cup \{z\}}\subset V(G)$, then by the induction hypothesis, $e(G-({N[x]\cup \{z\}}))\le \frac{15}{7}(n-7)$. Thus, by Lemma \ref{C-edge}, $e(G)\leq \frac{15}{7}n$.{ \hfill $\Box$}

\begin{fact}\label{lem-54star}
	If $G$ contains any copy of $5$-$4$ star, then $e(G)\leq \frac{15}{7}n$.
\end{fact}

\noindent{\bf Proof of Fact \ref{lem-54star}.}
Let $x\in V(G)$ with $d(x)=5=h+k$, and let $x_i\in N(x)$ with $d(x_i)=4$ for $1\le i\le 5$. Then by Lemma \ref{obs-star}, we have $N_2(x)\neq\emptyset$. Let $N_2(x)=\{y_{1},y_{2},\dots,y_t\}$ ($t\ge 1$). If $d(y_j)=5$ for some $1\le j\le t$, then $xx_iy_j$ is a $5$-$4$-$5$ path, and then  by Fact \ref{lem-5k5}, $e(G)\le \frac{15}{7}n$. So we may assume that $3\le d(y_j)\le 4$ for all $1\le j\le t$.

Let $Y=\{y_j~|~d(y_j)=4, 1\le j\le t\}$. We claim that $|Y|\le 1$. Suppose not, let $y_1,y_2\in N_2(x)$ with $d(y_1)=d(y_2)=4$, then by Lemma  \ref{obs-xyz}, $N(y_1), N(y_2)\subseteq N(x)$, it follows that $s:=|N(y_1)\cap N(y_2)|\ge 3$ as $d(x)=5$, and thus there is a $K_{3,s}$ in $G$ with the bipartition $\{x,y_1,y_2\}$ and $N(y_1)\cap N(y_2)$, a contradiction.  By Lemma \ref{obs-xyz}, $G[N[x]\cup N_2(x)]$ is a component of $G$. So we have $$e(G[N[x]\cup N_2(x)])\le \frac{1}{2}(5+20+4|Y|+3(t-|Y|))=\frac{25+3t+|Y|}{2}\le \frac{26+3t}{2}.\eqno(7)$$ Therefore, by the induction hypothesis and (7), we obtain that
\begin{align*}
	e(G)&=e(G-({N[x]\cup N_2(x)}))+e(G[N[x]\cup N_2(x)])\\
	&\le \frac{15}{7}(n-6-t)+  \frac{26+3t}{2}=\frac{15}{7}n-\frac{9t-2}{14}<\frac{15}{7}n
\end{align*}
as desired, where the last inequality follows from $t\ge 1$.{ \hfill $\Box$}\medskip

Recall that $V_i=\{v ~|~ v\in V(G)\ \text{and} \ d(v)=i\}$ and $n_i=|V_i|$. By Fact \ref{lem-s14}, we can assume that $V_5$ is an independent set in $G$. And by Fact \ref{lem-5k5}, we can assume that  $G$ contains no $5$-$l$-$5$ path for $3\le l\le 4$, that is, every pair of $V_5$ has no common neighbors. On the other hand,  by Fact \ref{lem-54star}, we can assume that $G$ contains no $5$-$4$ star, which implies that each vertex of $V_5$ is adjacent to at least one vertex of $V_3$. Then there exists an injective map from $V_5$ to $V_3$. It follows that $n_{5} \leq n_{3}$.\medskip

Note that $\delta\ge 3$ and $\Delta=5$, then $n=\sum\limits_{i=3}^{5}n_{i}$, and thus
\begin{align*}
	e(G)&=\frac{1}{2}\sum\limits_{v\in V(G)}d(v)=\frac{1}{2}\sum\limits_{i=3}^{5}i\cdot n_{i}=\frac{1}{2}(5n_{5}+3n_{3}+4(n-n_{5}-n_{3}))\\
	&=\frac{1}{2}(4n + n_{5}-n_{3})\leq 2n < \frac{15}{7}n.
\end{align*}


\subsection{$h+k=6$}

In this subsection, $(h,k)\in \{(1,5), (2,4)\}$, $f(n)=\frac{5}{2}n$, $\delta(G)\ge3$ and $\Delta(G)\le 6$. If $\Delta(G)\le 5$, then $e(G) \leq \frac{5}{2}n$; and if $n\le 12$, then by Lemma \ref{thm-edge}, $e(G) \leq3n-6\le\frac{5}{2}n$ clearly. So we assume that $n\ge 13$ and $\Delta(G)= 6$.

\begin{fact}\label{lem-s15}
	
If $G$ contains any copy of $6$-$6$ edge, then $e(G)\leq \frac{5}{2}n$.
	
\end{fact}

\noindent{\bf Proof of Fact \ref{lem-s15}.}
Let $xy$ be the $6$-$6$ edge in $G$, then by Claim \ref{S-com}, $S[xy]$ is a component of $G$. Note that $$|S[xy]|=12- |S_{xy}|\le12=\frac{6}{3-5/2}< 13\le n~ (\tau:=5/2),$$ and hence $S[xy]\subset V(G)$. By the induction hypothesis, $e(G-{S[xy]})\le \frac{5}{2}(n-|S[xy]|)$. So $e(G)\leq\frac{5}{2}n$ follows from Lemma \ref{C-edge}.
{ \hfill $\Box$}
	
\begin{fact}\label{lem-6k6path}
	 If $G$ contains any copy of $6$-$l$-$6$ path for $3\le l\le 5$, then $e(G)\leq \frac{5}{2}n$.	
\end{fact}

\noindent{\bf Proof of Fact \ref{lem-6k6path}.}
Let $xyz$ be a $6$-$l$-$6$ path in $G$ with $d(x)=d(z)=6=h+k$ and $d(y)=l$ for $3\le l\le 5$. By Fact \ref{lem-s15}, we can suppose that $xz\notin E(G)$. Then by Claim  \ref{S-3path}, $N(x)=N(z)$ and $N[x]\cup \{z\}$ is a component of $G$. Note that $|N[x]\cup \{z\}|=8<\frac{6}{3-5/2}=12<13\le n$  ($\tau:=5/2$), that is $N[x]\cup \{z\} \subset V(G)$, and then by the induction hypothesis, $e(G-({N[x]\cup \{z\}}))\le \frac{5}{2}(n-8)$. Thus, by Lemma \ref{C-edge}, $e(G)\leq \frac{5}{2}n$.{ \hfill $\Box$}

\begin{fact}\label{lem-65star}
If $G$ contains any copy of $6$-$5$ star, then $e(G)\leq \frac{5}{2}n$.	
\end{fact}
	
\noindent{\bf Proof of Fact \ref{lem-65star}.} Let $x\in V(G)$ with $d(x)=h+k=6$ and let $x_i\in N(x)$ with $d(x_i)=5$ for $1\le i\le6$. Then  by Lemma \ref{obs-star}, we have $N_2(x)\neq\emptyset$, and then by Lemma \ref{obs-xyz}, $N[x]\cup N_2(x)$ is component of $G$. Let $N_2(x)=\{z_1,z_2,\dots,z_s\}$, where $s\ge1$. Then $2e(G[N[x]\cup N_2(x)])=d(x)+\sum_{i=1}^5d(x_i)+\sum_{j=1}^sd(z_j)=36+\sum_{j=1}^sd(z_j)$. If $d(z_j)=6$, then $G$ contains a $6$-$5$-$6$ path, and thus by Fact \ref{lem-6k6path}, $e(G)\le \frac{5}{2}n$. So we may assume that $3\le d(z_j)\le 5$ for all $1\le j\le s$. Let $Z=\{z_j~|~d(z_j)=5,1\le j\le s\}$.

If $Z=\emptyset$, that is $ d(z_j)\le 4$ for $1\le j\le s$, then
$$e(G[N[x]\cup N_2(x)])\le 18+2s \le\frac{5(7+s)}{2},\eqno(8)$$ where the last inequality follows from $s\ge1$.
	
If $Z\neq\emptyset$, then we first claim that $|Z|=1$ (for otherwise, let $z_1,z_2\in Z$, then by Lemma  \ref{obs-xyz}, we have $N(z_j)\subseteq N(x)$ for $1\le j\le2$, and thus $a:=|N(z_1)\cap N(z_2)|\ge 4$, which implies that $G$ contains a $K_{3,a}$ with the bipartition $N(z_1)\cap N(z_2)$ and $\{x,y_1,y_2\}$, a contradiction). So we can assume $Z=\{z_1\}$ and $N(z_1)=\{x_1,\ldots,x_5\}$. Next, we claim that $s\ge2$ (for otherwise, since $N(x_6)\setminus\{x\}\subset\{x_1,\ldots,x_5\}$ and $b:=|N(x)\cap N(x_6)|=4$ as $d(x_6)=5$, $G$ contains a $K_{3,b}$ with the bipartition  $\{x,x_6,z_1\}$ and $N(x)\cap N(x_6)$, a contradiction). Thus we can obtain that $$e(G[N[x]\cup N_2(x)])\le \frac{36+5+4(s-1)}{2}=\frac{37+4s}{2}\le\frac{5(7+s)}{2}\eqno(9)$$ for $s\ge2$.

Recall that $N[x]\cup N_2(x)$ is a component of $G$, then by the induction hypothesis, together with (8) and (9), we have \begin{align*}
e(G)&=e(G-({N[x]\cup N_2(x)}))+e(G[N[x]\cup N_2(x)])\\
&\le \frac{5}{2}(n-7-s)++e(G[N[x]\cup N_2(x)])\\
&\le\frac{5}{2}(n-7-s)+\frac{5}{2}(7+s)=\frac{5}{2}n.
\end{align*}Therefore, the proof of Fact \ref{lem-65star} is complete.
\hfill $\Box$ \medskip 

Recall that $V_i=\{v ~|~ v\in V(G)\ \text{and} \ d(v)=i\}$ and $n_i=|V_i|$. Then $V_i\neq \emptyset$ for $3\le i\le 6$ and $V_i=\emptyset$ otherwise. By Fact \ref{lem-s15},  we may assume that $V_6$ is an independent set in $G$. And by Fact \ref{lem-6k6path}, we can assume that  $G$ contains no $6$-$l$-$6$ path for $3\le l\le5$, which implies that every pair of $V_6$ has no common neighbors. On the other hand, we may suppose that $G$ contains no $6$-$5$ star by Fact \ref{lem-65star}, that is, each vertex of $V_6$ is adjacent to at least one vertex of $V_4\cup V_3$. It implies that there exists an injective map from $V_6$ to $V_4\cup
V_3$. Thus we have  $n_{6} \leq n_4+n_{3}$.\medskip

Note that $n=\sum\limits_{i=3}^{6}n_{i}$, then
\begin{align*}
	e(G)&=\frac{1}{2}\sum\limits_{i=3}^{6}i\cdot n_{i}=\frac{1}{2}(6n_{6}+4n_{4}+3n_{3}+5(n-n_6-n_{4}-n_{3}))\\
	&=\frac{1}{2}(5n + n_{6}-n_{4}-2n_{3})\leq \frac{5}{2}n.
\end{align*}

\subsection{$h=2,~k=5$}

In this subsection, we will show that $\ex_{\p}(n,W_{2,5})\le \frac{17}{6}n$.

Recall that $\delta(G)\ge3$ and $\Delta(G)\le7$. If  there is a vertex  $x\in V(G)$ such that $d(x)=7=h+k$, then by Lemma  \ref{obs-xyz}, $N(u)\subseteq N(x)$ for $u\in N_2(x)$, and $N[x]\cup N_2(x)$ is a component of $G$. So $|N(u)\cap N(x)|=d(u)\ge 3$ for $u\in N_2(x)$. Therefore $3|N_2(x)| \le e[N(x),N_2(x)]\le d(x)(d(x)-1)=42,$ which implies that $|N_2(x)|\le14$. Thus $$|N[x]\cup N_2(x)|=1+d(x)+|N_2(x)|\le22<36=\frac{6}{3-17/6}\le n,$$ then by Lemma \ref{C-edge}, together with the induction hypothesis,
$e(G)\le\frac{17}{6}n.$ Hence we have the following assumption.

\begin{assm}\label{d=6}
	$\Delta(G)\le 6$.
\end{assm}

If $d(x)+d(y)\le11$ for all $xy\in E(G)$, then by Cauchy Inequality, we have $$11e(G)\ge \sum_{xy\in E(G)}(d(x)+d(y))=\sum_{x\in V(G)}(d(x))^2\ge \frac{(\sum_{x\in V(G)}d(x))^2}{n}=\frac{4(e(G))^2}{n},$$ that is $e(G)\le \frac{11}{4}n< \frac{17}{6}n$. Therefore we have the following assumption. 

\begin{assm}\label{xy=12}
	There exists an edge $xy\in E(G)$ such that $d(x)+d(y)\ge12$.
\end{assm}

By Assumptions \ref{d=6} and \ref{xy=12}, $G$ contains a $6$-$6$ edge $xy$. 

\begin{claim}\label{fact-2}
$|N(u)\cap (S_x\cup \overline{S[xy]})|\le1$ for each $u\in S_x$ and $|N(v)\cap (S_y\cup \overline{S[xy]})|\le1$ for each $v\in S_y$. Furthermore, $e[S_x,S_x\cup\overline{S[xy]}]\le |S_x|$ and $e[S_y,S_y\cup\overline{S[xy]}]\le |S_y|$.
\end{claim}

\noindent{\bf Proof of Claim \ref{fact-2}.}
If not, without loss of generality, we can suppose that there exists some $u\in S_x$ such that  $u_1,u_2\in N(u)\cap (S_x\cup \overline{S[xy]})$, then $u_1,u_2\notin N[y]$, and then  $(u_1u_2)u$-$x$-$y(y_1y_2\dots y_5)$ is a copy of  $W_{2,5}$, where $y_i\in N(y)$ for $1\le i\le 5$, a contradiction. So, $|N(u)\cap (S_x\cup \overline{S[xy]})|\le1$ for each $u\in S_x$, which implies $e[S_x,S_x\cup\overline{S[xy]}]\le |S_x|$.  \hfill $\Box$\medskip

Note that $|S[xy]|=12-|S_{xy}|$ and $|S_x|=|S_y|=5-|S_{xy}|$. By Lemma \ref{thm-edge}, $e[S_x\cup\{y\},S_y\cup\{x\}]\le2(|S_x|+1+|S_y|+1)-4= 4(5-|S_{xy}|)$. Since $\Delta(G)\le 6$, we have $e[S_{xy},S_x\cup S_y\cup \overline{S[xy]}]\le (6-2)|S_{xy}|=4|S_{xy}|$. Hence, by Claim  \ref{fact-2} and the induction hypothesis, we have
\begin{align*}
~~~~~~~~~~~e(G)&=e(G-{S[xy]})+e(G[S[xy]])+e[S[xy],\overline{S[xy]}]\\
	&=e(G-{S[xy]})+e[S_x\cup\{y\},S_y\cup\{x\}]+e[S_{xy},V(G)]\\
	&~~~+e[S_x,S_x\cup\overline{S[xy]}]+e[S_y,S_y\cup\overline{S[xy]}]\\
	&\leq \frac{17}{6}(n-12+|S_{xy}|)+4(5-|S_{xy}|)+6|S_{xy}|+2(5-|S_{xy}|)\\
	&=\frac{17}{6}n+\frac{17}{6}|S_{xy}|-4.~~~~~~~~~~~~~~~~~~~~~~~~~~~~~~~~~~~~~~~~~~~~~~~~~~~~~~~~	~~(10)
\end{align*}
If $|S_{xy}|\le 1$, then by $(10)$, $e(G)< \frac{17}{6}n$. So, in the remainder, we can assume that $2\le|S_{xy}|\le5$. Then $|S_y|\le 5-|S_{xy}|\le 3$.\medskip

Denote $S':=(N_2(x)\cup N_2(y))\cap \overline{S[xy]}$ and $S'_4:=\{ v\in S' \ | \ d(v)=4\}$.

\begin{claim}\label{fact-n}
	Let $v\in S'$. Then we have the following:
	
(i)  ~~~$|N(v)\cap {S_{xy}}|\le 2$;

(ii)  ~~$|N(v)\cap{\overline{S[xy]}}|\le1$;

(iii) ~If $N(v)\cap S_{xy}=\emptyset$, then $|N(v)\cap S_x|\ge 2$ or $|N(v)\cap S_y|\ge 2$;

(iv) ~If $|N(v)\cap S_{xy}|=1$, then $|N(v)\cap S_{x}|=|N(v)\cap S_{y}|=1$;

(v) ~~If $v\in N_2(x)$, then $|N(v)\cap S_y|\le 1$; and if $v\in N_2(y)$,  then $|N(v)\cap S_x|\le 1$. 
\end{claim}

\noindent{\bf Proof of Claim \ref{fact-n}.} (i)  If $|S_{xy}|\le 2$, then $|N(v)\cap {S_{xy}}|\le 2$ clearly. So we assume $|S_{xy}|\ge 3$. If $|N(v)\cap {S_{xy}}|\ge 3$, then there is a $K_{3,a}$ with the bipartition $\{x,y,v\}$ and  $N(v)\cap S_{xy}$, where $a:=|N(v)\cap {S_{xy}}|$, a contradiction. Therefore, $|N(v)\cap {S_{xy}}|\le 2$.
\medskip

(ii) Suppose that $|N(v)\cap{\overline{S[xy]}}|\ge2$ for some $v\in S'$, say $v_1,v_2\in N(v)\cap {\overline{S[xy]}}$. Since $v\in S'$, we can assume, without loss of generality, that $x_1\in N(v)\cap N(x)$. Note that $v_1, v_2\notin N[x]$, and then $(v_1v_2)v$-$x_1$-$x(yx_2\dots x_5)$ is a copy of  $W_{2,5}$, where $x_i\in N(x)\setminus\{y\}$ for $2\le i\le 5$, a contradcition. \medskip

(iii) If not, then $|N(v)\cap S_x|\le 1$ and $|N(v)\cap S_y|\le1$ for $v\in S^{'}$. Since $N(v)\cap S_{xy}=\emptyset$ and $|N(v)\cap{\overline{S[xy]}}|\le 1$ by (ii), we have $$3\le d(v)\le |N(v)\cap S_x|+|N(v)\cap S_y|+|N(v)\cap{\overline{S[xy]}}|\le 1+1+1=3$$
which implies that all the above inequalities should be equal, that is  $|N(v)\cap S_x|=|N(v)\cap S_y|=|N(v)\cap{\overline{S[xy]}}|=1$. Let $x_1\in N(v)\cap S_x$, $y_1\in N(v)\cap S_y$ and $w\in N(v)\cap{\overline{S[xy]}}$. Then
 $(y_1w)v$-$x_1$-$x(yx_2\dots x_5)$ is a copy of  $W_{2,5}$, where $x_i\in N(x)$ for $2\le i\le 5$, a contradiction.\medskip

(iv) Let $N(v)\cap S_{xy}=\{w\}$. We first show that $|N(v)\cap S_x|\le 1$, for otherwise, there are two vertices $x_1,x_2\in N(v)\cap S_x$ such that $(x_1x_2)v$-$w$-$y(xy_1\dots x_4)$ is a copy of  $W_{2,5}$, where $y_i\in N(x)$ for $1\le i\le 4$, a contradiction. Similarly, $|N(v)\cap S_y|\le 1$. By (ii), $|N(v)\cap{\overline{S[xy]}}|\le1$. If $N(v)\cap S_x=\emptyset$, then $$3\le \delta(G)\le d(v)=|N(v)\cap{\overline{S[xy]}}|+|N(v)\cap S_y|+|N(v)\cap S_{xy}|\le 1+1+1=3,$$ which implies that the above inequalities should be equal, i.e. $|N(v)\cap{\overline{S[xy]}}|=1$ and $|N(v)\cap S_y|=1$. Let  $v_1\in N(v)\cap{\overline{S[xy]}}$ and $y_1\in N(v)\cap S_y$, then $v_1,y_1\notin S_x$, and then $(v_1y_1)v$-$w$-$x(yx_1\dots x_4)$ is a copy of  $W_{2,5}$, where $x_j\in N(x)$ for $1\le j\le 4$, a contradiction. Thus, $|N(v)\cap S_x|=1$. Similarly, $|N(v)\cap S_y|=1$.\medskip

(v) By symmetry, we let $x_1\in N(v)\cap N(x)$. If $|N(v)\cap S_y|\ge 2$, say $y',y''\in N(v)\cap S_y$, then $y',y''\notin N[x]\setminus\{x_1\}$, and then $(y'y'')v$-$x_1$-$x(yx_2\dots x_5)$ is a copy of  $W_{2,5}$, where $x_i\in N(x)$ for $2\le i\le 5$, a contradiction. So $|N(v)\cap S_y|\le 1$.  \hfill $\Box$

\begin{claim}\label{fact-s251}
$d(v)\le 4$ for all $v\in S'$. Moreover, $|S'_4|\le |S_y|$.
\end{claim}

\noindent{\bf Proof of Claim \ref{fact-s251}.}
Let $v\in S'$. First we assume $N(v)\cap {S_{xy}}=\emptyset$. If $N(v)\cap S_x=\emptyset$ or $N(v)\cap S_y=\emptyset$, then $|N(v)\cap S_x|+|N(v)\cap S_y|\le |S_y|=|S_x|\le3$; and if $N(v)\cap S_x\not=\emptyset$ and $N(v)\cap S_y\not=\emptyset$, that is $v\in N_2(x)\cap N_2(y)$, then by Claim \ref{fact-n}(v), $|N(v)\cap S_y|\le 1$ and $|N(v)\cap S_x|\le 1$. Hence, by Claim \ref{fact-n}(ii),
\begin{align*}
	~~~d_G(v)&=|N(v)\cap S_x|+|N(v)\cap S_y|+|N(v)\cap \overline{S[xy]}|\\
	&\le
	\left\{
	\begin{array}{ll}
		3+1=4, &  \text{if } v\in N_2(x)\setminus N_2(y) \text{~or } v\in N_2(y)\setminus N_2(x);\tag{11}  \\
		1+1+1=3,  &  \text{if } v\in N_2(x)\cap N_2(y).
	\end{array}
	\right.
\end{align*}
Next we assume that $N(v)\cap {S_{xy}}\not=\emptyset$. Then by Claim \ref{fact-n}(i), $|N(v)\cap {S_{xy}}|\le 2$. If $N(v)\cap{\overline{S[xy]}}=\emptyset$, then by Claim \ref{fact-2}, $|N(v)\cap S_x|\le 1$ and $|N(v)\cap S_y|\le 1$, and thus

$$d(v)=|N(v)\cap {S_{xy}}|+|N(v)\cap S_x|+|N(v)\cap S_y|\le 2+1+1=4;\eqno(12)$$
and if $N(v)\cap{\overline{S[xy]}}\not=\emptyset$, then by Claim \ref{fact-2}, $|N(v)\cap{\overline{S[xy]}}|\le 1$, and then $ N(v)\cap (S_x\cup S_y)=\emptyset$. Hence $d(v)=|N(v)\cap {S_{xy}}|+|N(v)\cap \overline{S[xy]}|\le 2+1=3.$ 

Moreover, if $d(v)=4$, then by (11), $|N(v)\cap S_x|=3$ and $N(v)\cap S_y=\emptyset$, or $|N(v)\cap  S_y|=3$ and $N(v)\cap S_x=\emptyset$, see Figure \ref{f-d(v)=4}(i), that is, $|S_4^{'}|\le 2<3=|S_y|$; and by (12), $|N(v)\cap S_x|=|N(v)\cap S_y|=1$, see Figure \ref{f-d(v)=4}(ii). Hence $|S_4^{'}|\le|S_y|$ .\hfill $\Box$

\begin{figure}[ht]\medskip
	\centering
	\subfigure{
		\begin{minipage}[t]{0.3\textwidth}
			\centering
			\begin{tikzpicture}[scale=0.7]
				\tikzstyle{vertex}=[draw,circle,fill=black,minimum size=6,inner sep=0]		
				
				\node[vertex] (k) at
				(-1.3,-0.5){};
				\node[vertex] (l) at
				(3.7,-0.5){};
				\node[vertex] (e) at
				(3.7,0){};
				\node[vertex] (f) at
				(3.7,-1){};
				\node[vertex] (d) at (-1.3,-1) {};
				\node[vertex] (c) at (-1.3,0){};
				\node[vertex] (b) at (1.2,0) {};
				\node[vertex] (a) at (1.2,1){};
				\node[vertex] (g) at (0,2.8){};
				\node[vertex] (h) at (2.4,2.8){};
				\node[vertex] (i) at (0,4.0){};
				\node[vertex] (j) at (2.4,4.0){};
				\node[vertex] (x) at (0,-1)[label={[xshift=-6pt, yshift=-22pt] $x$}] {};
				\node[vertex] (y) at (2.4,-1)[label={[xshift=6pt, yshift=-22pt] $y$}] {};
				
				\draw (-2.4,-0.5) node[right=0pt]{$S_x$};
				\draw (3.7,-0.5) node[right=0pt]{$S_y$};
				\draw (0.6,1.4) node[right=0pt]{$S_{xy}$};
				\draw (3.5,2.8) node[right=0pt]{$S^{'}$};
				\draw (0,2.8) node[right=0pt]{$v$};
				\draw (2.4,2.8) node[right=0pt]{$v$};
				
				\draw[ultra thick] (x) -- (a);
				\draw[ultra thick] (x) -- (b);
				\draw[ultra thick] (x) -- (c);
				\draw[ultra thick] (x) -- (d);
				\draw[ultra thick] (x) -- (k);
				\draw[ultra thick] (y) -- (a);
				\draw[ultra thick] (y) -- (b);
				\draw[ultra thick] (y) -- (e);
				\draw[ultra thick] (y) -- (f);
				\draw[ultra thick] (y) -- (l);
				\draw[ultra thick] (x) -- (y);
				\draw[dashed, red] (g) -- (i);
				\draw[dashed, blue] (h) -- (j);
				\draw[red] (90:2.8cm) arc (90:216:1.7cm);
				\draw[red] (90:2.8cm) arc (90:225:1.9cm);
				\draw[red] (90:2.8cm) arc (90:230:2.1cm);
				\draw[blue,shift={(2.4,0)}] (0:1.3cm) arc (0:50:3.7cm);
				\draw[blue,shift={(2.4,0)}] (-40:1.64cm) arc (-45:80:2.3cm);
				\draw[ blue,shift={(2.4,0)}] (-20:1.4cm) arc (-20:65:2.7cm);
				
				\draw[thick, dashed,] (0.6,-0.5)rectangle(1.8,1.8);
				\draw[thick, dashed, blue] (-0.9,0.5)rectangle(-2.2,-1.5);
				\draw[thick, dashed, red] (4.5,0.5)rectangle(3.2,-1.5);
				\draw[thick, dashed,] (-2.0,2.1)rectangle(4.3,3.6);
			\end{tikzpicture}
			
			\centerline{(i) $N(v)\cap {S_{xy}}=\emptyset$}
			
		\end{minipage}
	}
	\hspace{2cm}
	\centering
	\subfigure{
		\begin{minipage}[t]{0.3\textwidth}
			\centering
			\begin{tikzpicture}[scale=0.7]
				\tikzstyle{vertex}=[draw,circle,fill=black,minimum size=6,inner sep=0]		
				
				\node[vertex] (g) at
				(2.4,2.7){};
				\node[vertex] (e) at
				(4.9,0){};
				\node[vertex] (f) at
				(4.9,-1){};
				\node[vertex] (d) at (-0.1,-1) {};
				\node[vertex] (c) at (-0.1,0){};
				\node[vertex] (b) at (2.4,0) {};
				\node[vertex] (a) at (2.4,1){};
				\node[vertex] (x) at (1.2,-1)[label={[xshift=-6pt, yshift=-22pt] $x$}] {};
				\node[vertex] (y) at (3.6,-1)[label={[xshift=6pt, yshift=-22pt] $y$}] {};
				\draw (-0.9,-0.5) node[right=0pt]{$S_x$};
				\draw (4.9,-0.5) node[right=0pt]{$S_y$};
				\draw (1.8,1.4) node[right=0pt]{$S_{xy}$};
				\draw (4.7,2.8) node[right=0pt]{$S^{'}$};
				\draw (2.4,2.8) node[right=0pt]{$v$};
				\draw[fill=black] (-0.1,-0.3) circle (0.8pt);
				\draw[fill=black] (-0.1,-0.5) circle (0.8pt);
				\draw[fill=black] (-0.1,-0.7) circle (0.8pt);
				\draw[fill=black] (4.9,-0.3) circle (0.8pt);
				\draw[fill=black] (4.9,-0.5) circle (0.8pt);
				\draw[fill=black] (4.9,-0.7) circle (0.8pt);
				\draw[fill=black] (2.4,0.3) circle (0.8pt);
				\draw[fill=black] (2.4,0.5) circle (0.8pt);
				\draw[fill=black] (2.4,0.7) circle (0.8pt);
				
				\draw[ultra thick] (x) -- (a);
				\draw[ultra thick] (x) -- (b);
				\draw[ultra thick] (x) -- (c);
				\draw[ultra thick] (x) -- (d);
				\draw[ultra thick] (y) -- (a);
				\draw[ultra thick] (y) -- (b);
				\draw[ultra thick] (y) -- (e);
				\draw[ultra thick] (y) -- (f);
				\draw[ultra thick] (x) -- (y);
				\draw[red] (g) -- (c);
				\draw[red] (g) -- (e);
				\draw[red] (g) -- (a);
				\draw[thick, dashed,] (1.8,-0.5)rectangle(3.0,1.8);
				\draw[thick, dashed, blue] (0.5,0.5)rectangle(-0.8,-1.5);
				\draw[thick, dashed, red] (5.5,0.5)rectangle(4.2,-1.5);
				\draw[thick, dashed,] (-0.8,2.1)rectangle(5.5,3.6);
				\draw[red,shift={(1.4,1.4)}] (-55:1.7cm) arc (-55:55:1.6cm);
			\end{tikzpicture}
			{(ii) $N_{S_{xy}}(v)\neq\emptyset$}	
		\end{minipage}
	}			
	\caption{The cases when $d(v)=4$}\label{f-d(v)=4}
\end{figure}
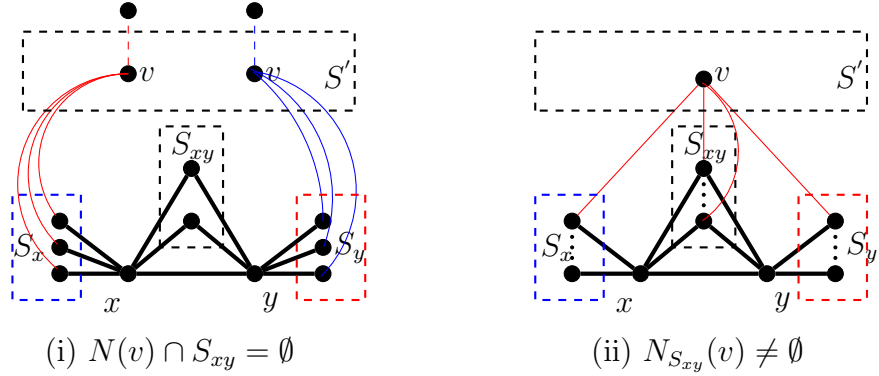
	\begin{claim}\label{fact-s252}
$|S^{'}|\le 12-\frac{3}{2}|S_y|$.
	\end{claim}

\noindent{\bf Proof of Claim \ref{fact-s252}.}
By Claim \ref{fact-n}(i), we have $|N(v)\cap S_{xy}|\le2$ for any $v\in S^{'}$. Let
\begin{align*}
	&S_0^{'}=\{v\ | \ v\in S^{'}\ \text{and} \ N(v)\cap S_{xy}=\emptyset \}, \\
	&S_1^{'}=\{v\ | \ v\in S^{'}\ \text{and} \ |N(v)\cap S_{xy}|=1\}, \\
	&S_2^{'}=\{v\ | \ v\in S^{'}\ \text{and} \ |N(v)\cap S_{xy}|=2\},
\end{align*}
 then $S^{'}= S_0^{'}\cup S_1^{'}\cup S_2^{'}$ and  $|S^{'}|= |S_0^{'}|+| S_1^{'}|+| S_2^{'}|$.

If $v\in S_0^{'}$, then by Claim \ref{fact-n}(iii), $|N(v)\cap S_x|\ge2$ and $|N(v)\cap S_y|\ge2$, which implies that $|S_0^{'}|\le2$ as $|S_x|=|S_y|=5-|S_{xy}|\le3$.

If $v\in S_1^{'}$, then by Claim \ref{fact-n}(iv), $|N(v)\cap S_{x}|=|N(v)\cap S_{y}|=1$, and thus $|S_1^{'}|\le |S_y|=|S_x|$.

On the other hand, $|S_1^{'}|+2|S_2^{'}|\le e(S_{xy},S^{'})\le (6-2)|S_{xy}|=4|S_{xy}|.$ So, $$|S_1^{'}|+|S_2^{'}|\le 2|S_{xy}|+\frac{1}{2}|S_1^{'}|\le 2|S_{xy}|+\frac{1}{2}|S_y|.$$ Hence, $|S^{'}|\le 2+2|S_{xy}|+\frac{1}{2}|S_y|=2+2(5-|S_{xy}|)+\frac{1}{2}|S_y|=12-\frac{3}{2}|S_y|.$
\hfill $\Box$\medskip

By Claim \ref{fact-s251}, we have $e[S^{'},\overline{S^{'}}]\le 3(|S^{'}|-|S_4^{'}|)+ 4|S_4^{'}|=3|S^{'}|+|S_4^{'}|\le 3|S^{'}|+ |S_y|.$ Since $G[S[xy]]$ is planar, $ e(G[S[xy]])\le 3(7+|S_y|)-6=15+3|S_y|$ by Lemma \ref{thm-edge}.
Note that $|S[xy]|=7+|S_y|$ and $|S_y|\le 3$, and hence, by the induction hypothesis, together with  Claim \ref{fact-s252}, we have
	\begin{align*}
		e(G)
		&=e(G-({S[xy]\cup S^{'}}))+e(G[S[xy]])+e[S^{'},\overline{S^{'}}]\\
		&\le \frac{17}{6}(n-7-|S_y|-|S^{'}|)+(15+3|S_y|)+(3|S^{'}|+ |S_y|) \\
		&=\frac{17}{6}n+\frac{|S^{'}|+7|S_y|-29}{6}\\
		&\le  \frac{17}{6}n+\frac{\frac{11}{2}|S_y|-17}{6}\le\frac{17}{6}n-\frac{1}{12} <\frac{17}{6}n.
	\end{align*}

This complete the proof of $\ex_{\p}(n,W_{2,5})\le\frac{17}{6}n$.

\bibliographystyle{abbrv}
\bibliography{myy2}
\end{sloppypar}	
\end{document}